\documentclass{article}
\usepackage[T1]{fontenc}
\usepackage[cp1250]{inputenc}
\usepackage{amsthm}
\usepackage{amsmath}
\usepackage{amsfonts}
\usepackage{multicol}
\usepackage{epsfig}
\usepackage{verbatim}
\usepackage{hyperref}
\usepackage[affil-it]{authblk}
\usepackage[authoryear]{natbib}

\newtheorem{theorem}{Theorem}
\newtheorem{lemma}{Lemma}
\newtheorem{corollary}{Corollary}
\theoremstyle{definition}
\newtheorem{example}{Example}
\theoremstyle{plain}

\newtheorem{proposition}{Proposition}

\def\T{\mathfrak T}
\def\diag{\mathrm{diag}}

\title{Optimal approximate designs for estimating treatment contrasts resistant to nuisance effects}
\author{Samuel Rosa and Radoslav Harman}
\affil{Faculty of Mathematics, Physics and Informatics, Comenius University, Bratislava, Slovakia}
\date{\today} 

\begin{document}
	
	\maketitle

\section*{Abstract}

Suppose that we intend to perform an experiment consisting of a set of independent trials. The mean value of the response of each trial is assumed to be equal to the sum of the effect of the treatment selected for the trial, and some nuisance effects, e.g., the effect of a time trend, or blocking.
In this model, we examine optimal approximate designs for the estimation of a system of treatment contrasts, with respect to a wide range of optimality criteria.

We show that it is necessary for any optimal design to attain the optimal treatment proportions, which may be obtained from the marginal model that excludes the nuisance effects.
Moreover, we prove that for a design to be optimal, it is sufficient that it attains the optimal treatment proportions and satisfies conditions of resistance to nuisance effects. 
For selected natural choices of treatment contrasts and optimality criteria, we calculate the optimal treatment proportions and give an explicit form of optimal designs. In particular, we obtain optimal treatment proportions for comparison of a set of new treatments with a set of controls.
The results allow us to construct a method of calculating optimal approximate designs with a small support by means of linear programming. As a consequence, we can construct efficient exact designs by a simple heuristic.

\section{Introduction}\label{sIntro}

The results of an experiment may be affected by conditions with effects that we aim to estimate and by other conditions with nuisance effects. For example, any
experiment that consists of multiple trials performed in a time sequence may be subject to a nuisance time trend caused by the ageing of the material used for the experiment, wearing down of the experimental devices, changes in the temperature, etc. Many agricultural experiments are subject to a two-dimensional nuisance trend, resulting from the arrangement of the trials in a two-dimensional field, see, e.g., \cite{JacrouxMajumdar} and \cite{BaileyWilliams}. The objective of the experimental design in such cases is to eliminate the nuisance effects, or to provide as much information as possible on the effects of interest.
\bigskip

The aim of this paper is to provide the $\Phi$-optimal approximate designs for estimating a system of contrasts of treatment effects under the presence of nuisance effects, where $\Phi$ is a given optimality criterion. Particularly, we aim to provide designs optimal under the presence of nuisance time trends. 

There is a large amount of literature on exact designs in such models, especially on block designs (e.g., \cite{MajumdarNotz}, \cite{Majumdar}, \cite{Jacroux02}), but also on trend resistant designs (e.g., \cite{Cox}, \cite{Cheng},  \cite{AtkinsonDonev}) or block designs in the presence of a trend (e.g., \cite{BradleyYeh}, \cite{JacrouxMajumdar}). However, these results are usually tailored for a particular model and a system of contrasts, often with limiting assumptions (e.g., the relationship between the number of blocks and treatments, or trends modelled by polynomials of given degrees).
In approximate theory, \cite{Pukelsheim83} studied optimal block designs for estimating centered contrasts, \cite{GiovagnoliWynn} obtained optimal block designs for comparing treatments with a control with respect to Kiefer's optimality criteria and \cite{Schwabe} studied product designs in additive models. 

The results on approximate designs are usually simpler and more general than the results on exact designs, therefore they provide a valuable insight into the qualitative aspects of the design problem. Moreover, the optimal approximate designs facilitate the computation of informative lower bounds on the efficiency of exact designs. However, it is not always clear how to convert approximate designs to exact designs that can be used for an actual, finite-size experiment.
In this paper, we provide conditions of approximate optimality of designs for the estimation of any system of contrasts in a general additive model with any system of nuisance effects. Moreover, we demonstrate that the conditions can be employed to construct efficient exact designs by means of linear programming.
\bigskip

We show that a $\Phi$-optimal approximate design may be obtained in two steps: (i) Calculate $\Phi$-optimal proportions of treatment replications (treatment weights). These optimal proportions depend on the choice of contrasts of interest and on the optimality criterion $\Phi$; however, they do not depend on the nuisance effects. (ii) Subject to keeping the optimal proportions of treatment replications, distribute the treatments to nuisance conditions such that the resulting design is \emph{resistant to nuisance effects}. The designs resistant to nuisance effects are an extension of the designs orthogonal to the time trend (balanced for trend, or trend-free, cf. \cite{Cox}, \cite{JacrouxRay}) to a more general class of models and treatment contrasts.

The approach of first finding a design in a simpler model and then assuring that the information is retained in a finer model was used, e.g., in \cite{Schwabe} and \cite{Kunert}. \cite{Schwabe} studied optimal product designs, unlike the present paper, where optimal designs with non-product structure are provided too; \cite{Kunert} studied  exact designs in the case of universal optimality. Universal optimality, formulated by \cite{Kiefer}, means optimality for estimating a maximal system of orthonormal contrasts (which is a special case of the general system of contrasts that we consider), with respect to a wide range of criteria.

For selected systems of treatment contrasts and a wide class of optimality criteria $\Phi$, we calculate $\Phi$-optimal treatment weights and thus obtain a class of $\Phi$-optimal designs. For instance, for the estimation of contrasts for comparing a set of new treatments with a set of controls, we provide $MV$-optimal designs and optimal designs with respect to Kiefer's $\Phi_p$-optimality criteria, $p \in [-\infty,0]$, including $A$- and $E$-optimality ($p=-1$ and $p=-\infty$, respectively). These results generalize the results given by \cite{GiovagnoliWynn}, who obtained $\Phi_p$-optimal block designs for comparing treatments with one control.
For any \emph{completely symmetric} system of contrasts, we show that the uniform design is $\Phi$-optimal for all orthogonally invariant information functions, which generalizes, for instance, a result given by \cite{Pukelsheim83}.

The obtained results may be used to analytically construct optimal approximate designs. A special case of $\Phi$-optimal designs resistant to nuisance effects are the product designs with $\Phi$-optimal treatment proportions (cf., e.g., \cite{Schwabe}), but the approximate product designs have a large support, which makes the transition to exact designs difficult. 
However, the set of optimal approximate designs is typically large and both the conditions of optimal treatment weights and the conditions of resistance to nuisance effects are linear. Therefore, we can employ the simplex method of linear programming to obtain optimal approximate designs with a small support. This allows us to construct efficient exact designs using a simple heuristic, especially in the presence of nuisance trends.

In the rest of Section \ref{sIntro}, our notation and the statistical model is established. The main theoretical results are proved in Section \ref{sBalanced}. In the same section, we provide optimal treatment proportions for estimating particular sets of contrasts. Examples of experiments under the presence of nuisance effects are provided in Section \ref{sExamples}. The theoretical results are applied in Section \ref{sConstructing} to obtain optimal approximate designs with small support and consequently efficient exact designs.

\subsection{Notation}

The symbols $1_n$ and $0_n$ denote the column vectors of length $n$ of ones and zeroes, respectively. The symbol $J_n$ denotes the $n \times n$ matrix $J_n = 1_n 1^T_n$ of ones and $e_u$ is the $u$-th standard unit vector (the $u$-th column of the identity matrix $I_n$, where $n$ is the dimension of $e_u$).
By the symbol $0_{m\times n}$, or by $0$ if the dimensions are clear from the context, we denote the $m\times n$ matrix of zeroes.
We denote the null space and the column space of a matrix $A$ by $\mathcal{N}(A)$ and $\mathcal{C}(A)$, respectively.
By the symbol $\mathfrak{S}^s_+$ we denote the set of $s \times s$ non-negative definite matrices and 
by $\preceq$ we denote the Loewner ordering of matrices in $\mathfrak{S}^s_+$, i.e., $A \preceq B$ if $B-A$ is non-negative definite.
Let $x=\big(x_1, \ldots, x_n\big)^T$ be a vector with non-zero components, then by $x^{-1}$ we denote the vector $x^{-1} := \big(x_1^{-1}, \ldots, x_n^{-1}\big)^T$.
By $\diag(v_1,\ldots,v_k)$, where $v_1,\ldots,v_k$ are column or row vectors, we denote the diagonal matrix with diagonal elements corresponding to the elements of $v_1,\ldots,v_k$.

\subsection{Statistical Model}

Consider an experiment consisting of $N$ trials, where in each trial we choose one of $v$ treatments ($v \geq 2$). The response of the $i$-th trial is determined by the  effect $\tau_{u(i)}$ of the chosen treatment $u(i)$ and by the effects of nuisance experimental conditions $t(i)$ from a finite set $\T$, $|\T| =: n < \infty$.

We assume that the model is additive in the treatment and nuisance effects and that it can be expressed as
\begin{equation}\label{eModel1}
	Y_i=\tau_{u(i)}+h^T(t(i)) \theta + \varepsilon_i, \quad i=1,\dots,N,
\end{equation}
where $Y_1, \ldots, Y_N$ are the observations, $\theta$ is a $d \times 1$ vector of nuisance parameters, $h: \mathfrak{T} \rightarrow \mathbb{R}^d$ is the regressor of the nuisance experimental conditions, and $\varepsilon_1, \dots, \varepsilon_N$ are independent and identically distributed random errors with zero mean and variance $\sigma^2 < \infty$. Suppose that we aim to estimate a system of $s$ contrasts $Q^T \tau$, where $\tau = \big(\tau_1, \ldots, \tau_v \big)^T$ and $Q$ is a $v \times s$ matrix satisfying $Q^T 1_v  = 0_s$. We will assume that $Q$ has full rank $s$, unless stated otherwise.  Moreover, we will assume that we are interested in all treatments $1, \ldots, v$, i.e., each treatment is present in $Q$ (no row of $Q$ is $0_{s}^T$). We consider $\theta$ to be a vector of nuisance parameters. 

The model (\ref{eModel1}) can be expressed in the linear regression form
$$Y_i = f^T(x_i)\beta + \varepsilon_i, \quad i=1,\ldots, N, $$
where $x_i = (u(i),t(i)) \in \mathfrak{X}$, $\mathfrak{X} = \{1,\ldots, v\} \times \mathfrak{T}$, $f(u,t) = \big(e_u^T, h^T(t) \big)^T$, $\beta = \big(\tau^T,\theta^T\big)^T$.
The objective of the experiment is to estimate a system of contrasts $K^T \beta$, where $K^T = \big(Q^T, 0_{s \times d} \big)$.

Let the \emph{approximate} design of experiment (or, in short, design) be a function $\xi: \mathfrak{X} \rightarrow [0,1]$, such that $\sum\limits_{x \in \mathfrak{X}} \xi(x) = 1$, where $\xi(x)$ represents the proportion of trials to be performed in $x \in \mathfrak{X}$. Hence, an \emph{exact} design of experiment of size $N$ is represented by a function $\xi: \mathfrak{X} \rightarrow \{0,1/N,2/N,\ldots,1\}$, such that $\sum\limits_{x \in \mathfrak{X}} \xi(x) = 1$, where $N\xi(x)$ is the number of trials in the design point $x \in \mathfrak{X}$.

The information matrix of the design $\xi$ for estimating $K^T\beta$ is the non-negative definite matrix (see \cite{puk})
\begin{equation}\label{eInfMat}
	N_K(\xi)= \mathrm{min}_{L \in \mathbb{R}^{s \times m}: LK=I_s} LM(\xi)L^T,
\end{equation}
where $M(\xi)=\sum_{x \in \mathfrak{X}} \xi(x)f(x)f^T(x)$ is the moment matrix of the design $\xi$ and the minimization is taken with respect to the Loewner ordering $\preceq$.
It is well known that the system $K^T \beta$ is estimable if and only if $\mathcal{C}(K) \subseteq \mathcal{C}(M(\xi))$. When $K^T\beta$ is estimable under $\xi$, we say that $\xi$ is feasible for $K^T\beta$. In such a case, the information matrix of $\xi$ is $N_K(\xi) = (K^T M^-(\xi) K)^{-1}$, where $M^-(\xi)$ is a generalized inverse of $M(\xi)$. 

Let $\Phi: \mathfrak{S}^s_+ \rightarrow \mathbb{R}$ be an optimality criterion. Then, a design $\xi^*$ is said to be $\Phi$\emph-optimal if it maximizes $\Phi\big(N_K(\xi)\big)$ among all feasible designs $\xi$. A widely used class of optimality criteria are the Kiefer's $\Phi_p$ criteria. Let $H$ be a positive definite $s \times s$ matrix  with eigenvalues $\lambda_1(H), \ldots, \lambda_s(H)$, and let $\lambda_\mathrm{min}(H)$ be the smallest eigenvalue of $H$. Then,
$$
\Phi_p(H)=
\begin{cases}
\; \Big(\frac{1}{s} \sum\limits_{j=1}^{s} \lambda_j^p(H) \Big)^{1/p}, & p \in (-\infty, 0), \\
\; \Big(\prod\limits_{j=1}^{s} \lambda_j(H) \Big)^{1/s}, & p=0, \\
\; \lambda_\mathrm{min}(H), & p=-\infty.
\end{cases}
$$
If $H$ is singular, we set $\Phi_p(H)=0$.
For $p=0$, $-1$ and $-\infty$, we obtain the $D$-, $A$- and $E$- optimality criterion, respectively. Note that $\Phi_p$ criteria are information functions (see \cite{puk}), in particular they are Loewner isotonic, positively homogeneous and concave.
\\

We will investigate further the properties of experimental designs in model (\ref{eModel1}). The moment matrix of a design $\xi$ may be expressed in the form
$$M(\xi)=
\begin{bmatrix}
M_{11}(\xi) & M_{12}(\xi) \\
M_{12}^T(\xi) & M_{22}(\xi)
\end{bmatrix},$$
where
\begin{eqnarray*}
	M_{11}(\xi)&=&\mathrm{diag}\left( \sum_{t \in \T}\xi(1,t), \dots, \sum_{t \in \T}\xi(v,t)  \right),\\
	M_{12}(\xi)&=&\left(\sum_{t \in \T}\xi(1,t)h(t), \ldots, \sum_{t \in \T}\xi(v,t)h(t) \right)^T,\\
	M_{22}(\xi)&=&\sum_{t \in \T} \left(\sum_{u=1}^{v}\xi(u,t) \right) h(t)h^T(t).
\end{eqnarray*}

Let us denote the Schur complement of the moment matrix $M(\xi)$ as $M_\tau(\xi) = M_{11}(\xi)-M_{12}(\xi) M^-_{22}(\xi) M_{21}(\xi)$. It is simple to show that the system $K^T\beta$ is estimable under a design $\xi$ if and only if $\mathcal{C}(M_\tau(\xi)) \subseteq \mathcal{C}(Q)$. If $K^T \beta$ is estimable under $\xi$, the information matrix of $\xi$ is $N_K(\xi)=(Q^T M_\tau^-(\xi) Q)^{-1}.$

\section{Optimal Approximate Designs}\label{sBalanced}

\subsection{Preliminaries}

We say that $w$ is a \emph{treatment proportions design} if it is a design in the marginal model without nuisance effects
\begin{equation}\label{eModelWithoutTrend}
	Y_i=\tau_{u(i)}+\varepsilon_i, \quad i=1,\dots,N.
\end{equation}
That is, $w$ is a function from $\{1,\ldots,v\}$ to $[0,1]$ satisfying $\sum_u w(u) = 1$. For a design $\xi$ of (\ref{eModel1}), the marginal design defined by $w(u) = \sum_t \xi(u,t)$ for all $u$ represents the total weights of individual treatments in $\xi$, and it will be called the \emph{treatment proportions design of $\xi$}. Since a design $w$ always provides $v$ weights $w(1), \ldots, w(v)$, we will often equivalently denote $w$ as a $v \times 1$ vector of weights $w = \big(w_1, \ldots, w_v \big)^T$. Note that if $\xi$ is an exact design of size $N$ and $w$ is its treatment proportions design, then $Nw$ is the vector of replications of treatments in $\xi$.

The properties of a treatment design $w$ in model (\ref{eModelWithoutTrend}) are generally very simple to analyze. For instance, it is straightforward to show that the moment matrix of $w$ is $M(w)=\mathrm{diag}(w)$. Moreover, the set of contrasts $Q^T \tau$ is estimable in (\ref{eModelWithoutTrend}) under $w$ if and only if $w_u > 0$ for all $u$. In such a case, the information matrix of $w$ is evidently $N_Q(w)=\big(Q^T\mathrm{diag}(w^{-1})  Q\big)^{-1}$.
\bigskip

Similarly to the treatment replications design, we define \emph{nuisance conditions design} $\alpha$ to be a design in the marginal model without treatments
\begin{equation}\label{eModelWithoutTreatments}
	Y_i=h^T(t(i))\theta+\varepsilon_i, \quad i=1,\dots,N,
\end{equation}
i.e., $\alpha$ is a function from $\T$ to $[0,1]$ that satisfies $\sum_t \alpha(t) = 1$. For a design $\xi$ of (\ref{eModel1}), the marginal design $\alpha(t) = \sum_u \xi(u,t)$ for all $t$ defines the proportions of trials to be performed under particular nuisance conditions, and it will be called the \emph{nuisance conditions design of $\xi$}.

\begin{proposition}\label{pNuisanceReducesInformation}
	Let $\xi$ be a design in model (\ref{eModel1}) and let $w$ be its treatment proportions design. Then, $N_K(\xi) \preceq N_Q(w)$.
\end{proposition}

The proof of Proposition \ref{pNuisanceReducesInformation} and all other proofs are deferred to the appendix.
The proposition shows that by introducing nuisance effects, the information about the contrasts of interests can not increase. However, for a large class of designs, the information is exactly retained.

We will say that a design $\xi$ with its treatment design $w$ is \emph{resistant to nuisance effects}, or \emph{nuisance resistant} for a given system of contrasts $Q$, if it satisfies
\begin{equation}\label{eNuisanceResistant}
	\begin{bmatrix}
		\frac{1}{w_1}\sum\limits_{t\in \T} \xi(1,t)h(t), &  \ldots, & \frac{1}{w_v}\sum\limits_{t\in \T} \xi(v,t)h(t)
	\end{bmatrix} Q = 0.
\end{equation}

The following proposition justifies this definition.

\begin{proposition}\label{pBalanceSameIM}
	Let $\xi$ be a nuisance resistant design with its treatment proportions design $w>0$, then
	(i) $\xi$ is feasible for $K^T\beta$,
	(ii) $K^T M^-(\xi)K = Q^T M^{-1}(w)Q$ and
	(iii) $\xi$ has the same information matrix as $w$, i.e., $N_K(\xi) = N_Q(w)$.
\end{proposition}

Note that the conditions of resistance to nuisance effects have also another desirable property: they are invariant with respect to a regular reparametrization of the nuisance regressors. That is, a design is resistant to nuisance effects with respect to nuisance regressors $h$ if and only if it is resistant to nuisance effects with respect to nuisance regressors $\tilde{h} = Rh$, where $R$ is any non-singular $d \times d$ matrix.
\bigskip

In general, the class of designs resistant to nuisance effects depends on the chosen system of contrasts $Q$. Nevertheless, as we show, there is a large subclass of nuisance resistant designs that is invariant to the choice of $Q$, i.e., these designs satisfy (\ref{eNuisanceResistant}) for any system of contrasts.

We will say that a design $\xi$ of (\ref{eModel1}) with its treatment design $w>0$ is \emph{balanced} if it satisfies
\begin{equation}\label{eBalanceCondGeneral2}
	\frac{1}{w_1}\sum_{t\in \T} \xi(1,t)h(t) = \frac{1}{w_2}\sum_{t\in \T} \xi(2,t)h(t) = \ldots = \frac{1}{w_v}\sum_{t\in \T} \xi(v,t)h(t).
\end{equation}

If $\xi$ is balanced, then for any $k \in \{1,\ldots,d\}$ the vector 
$$s_k:=
\Big(
\frac{1}{w_1}\sum_{t\in \T} \xi(1,t)h_k(t),  \ldots, \frac{1}{w_v}\sum_{t\in \T} \xi(v,t)h_k(t)
\Big)^T$$
satisfies $s_k=a_k1_v$ for some $a_k \in \mathbb{R}$.
Since $Q$ is a matrix of contrasts, we have $1_v^TQ=0_s^T$ and hence $s_k^T Q=0_s^T$. It follows that a balanced design $\xi$ is indeed nuisance resistant. 

When the matrix of contrasts $Q$ attains the maximum rank, $v-1$, the null space $\mathcal{N}(Q^T)$ has dimension $1$ and it consists of vectors of the form $a1_{v}$ for $a \in \mathbb{R}$. Hence, for such $Q$, the balanced designs are the only nuisance resistant designs; i.e., in this specific but frequent case, the notions of  resistance to nuisance effects and balancedness coincide.
One consequence of this observation is that for given nuisance regressors $h$, the class of balanced designs is the intersection of the sets of nuisance resistant designs with respect to all possible choices of contrast matrices $Q$.

We remark that the conditions (\ref{eBalanceCondGeneral2}) mean that a design $\xi$ is balanced with respect to the nuisance effects. That is, for each regressor $h_k$, the weighted average of the values $h_k(t)$ with weights $\xi(u,t)/w_u$, $t \in \T$, is the same for each treatment $u$. Balancedness can also be understood geometrically: assume that for each $u$ we calculate the  barycentre of the $n$ points $h(t) \in \mathbb{R}^d$ with weights $\xi(u,t)/w_u$, $t \in \T$. Then, these barycentres must be the same for all treatments $u$.
\bigskip

A typical experimental situation is that we need to perform the same number of trials, usually one, under each nuisance condition $t \in \T$. In this case, it is straightforward to show that the conditions (\ref{eBalanceCondGeneral2}) may be expressed in a more compact form as follows.

\begin{proposition}\label{pBalanceCond}
	Let $\xi$ be a design which assigns the same weight to each nuisance condition, i.e., the nuisance conditions design of $\xi$ is $\alpha = 1_n/n$. Then, $\xi$ satisfies (\ref{eBalanceCondGeneral2}) if and only if it satisfies
	\begin{equation}\label{eBalanceCondGeneral3}
		\frac{1}{w_u}\sum_{t\in \T} \xi(u,t)h(t)=\frac{1}{n}\sum_{t\in \T} h(t) \text{ for all } u \in \{1,\ldots,v\}.
	\end{equation}
\end{proposition}

In the case of an exact design $\xi$ assigning one trial to each nuisance condition, the balancedness of $\xi$ has a straightforward interpretation. Suppose, for instance, that the nuisance conditions represent time and $h_1(t)$ is proportional to the room temperature at time $t$. For each treatment $u$, let $T_u$ be the average temperature at the times of trials with the treatment $u$. Then, the balance conditions for $h_1(t)$ mean that the temperature conditions are ``fair'' for all treatments in the sense that the average temperatures $T_u$ are the same: $T_u \equiv T$ for all $u$.
\bigskip

Let $w$ be a treatment proportions design and $\alpha$ be a nuisance conditions design. Then, a design $\xi$ is the \emph{product design} of $w$ and $\alpha$ if it satisfies
\begin{equation*}
	\xi(u,t)= w(u) \alpha(t) \text{ for all } u \in \{1,\ldots,v\}, t \in \T,
\end{equation*}
which we denote $\xi = w \otimes \alpha$. Note that any product design $w \otimes \alpha$ satisfies $\frac{1}{w_u}\sum_t \xi(u,t)h(t) = \sum_t \alpha(t) h(t)$ for all $u$, therefore the product design is balanced and consequently, it is also resistant to nuisance effects.

\subsection{Conditions of Optimality}

The following theorem shows that the optimality of treatment proportions is a necessary condition of the optimality of a design in model (\ref{eModel1}).

\begin{theorem}\label{tOptWeights}
	Let $\Phi$ be an information function, let $\xi^*$ be a $\Phi$-optimal design in model (\ref{eModel1}) and let $w^*$ be its treatment proportions design. Then, $w^*$ is $\Phi$-optimal in (\ref{eModelWithoutTrend}).
\end{theorem}

From Theorem \ref{tOptWeights} it follows that in order to find an optimal approximate design, we need to break up this process into two steps:
obtain the optimal treatment weights and then optimally allocate these weights to nuisance conditions. Note that finding a $\Phi$-optimal treatment design is a convex $v$-dimensional optimization problem
\begin{equation}\label{eMaxWeights}
	\max_{w>0, \, 1_v^Tw=1}\, \Phi((Q \diag(w^{-1})Q^T)^{-1}),
\end{equation}
which can usually be easily solved numerically, and often analytically, as we demonstrate in Subsection \ref{ssContrasts}.

Once the optimal treatment weights are obtained, we may allocate these weights to nuisance conditions using the following theorem, i.e., by choosing a nuisance resistant design.

\begin{theorem}\label{tBalanceOpt}
	Let $w^*$ be a $\Phi$-optimal treatment proportions design. Let $\xi^*$ be a nuisance resistant design with its treatment proportions design $w^*$. Then, $\xi^*$ is $\Phi$-optimal and it has the same information matrix as $w^*$, i.e., $N_K(\xi^*) = \big(Q^T \diag((w^*)^{-1}) Q\big)^{-1}$.
\end{theorem}

The balanced designs are nuisance resistant, therefore, the balanced designs with $\Phi$-optimal treatment weights $w^*$ are $\Phi$-optimal. Moreover, they have the same information matrix as $w^*$. Note that the set of optimal balanced designs is never empty, because it contains the set of product designs $w^* \otimes \alpha$ with any $\alpha$. Since $\alpha$ is any nuisance conditions design, the class of $\Phi$-optimal designs for model (\ref{eModel1}) is very large (unless $n=1$).

Similar results on optimality of product designs are given by \cite{Schwabe} (cf. Theorem 3.2) in a general additive model $Y_i = \beta_0 + f_1^T(u_1(i))\beta_1 + f_2^T(u_2(i))\beta_2 + \varepsilon_i$. Note that general nuisance resistant designs, because they need not have product structure, are not covered by \cite{Schwabe}.
\bigskip

Theorem \ref{tOptWeights} provides necessary conditions of optimality and Theorem \ref{tBalanceOpt} provides sufficient conditions of optimality.
It turns out that for the wide class of strictly concave optimality criteria, we can provide conditions that are both necessary and sufficient for optimality of a design $\xi$ in model (\ref{eModel1}). 

\begin{theorem}\label{tSuffAndNecc}
	Let $\Phi$ be a strictly concave information function. Then, a design $\xi$ is $\Phi$-optimal if and only if (i) its treatment proportions design $w$ is $\Phi$-optimal in model (\ref{eModelWithoutTrend}) and (ii) $\xi$ is resistant to nuisance effects.
\end{theorem}

Since the balanced designs are the only nuisance resistant designs for a system of contrasts of maximum rank, $v-1$, we obtain the following corollary.

\begin{corollary}\label{cBalancedAllOptimal}
	Let $\Phi$ be a strictly concave information function, and let $Q$ be a matrix of contrasts of rank $v-1$. Then, a design $\xi$ is $\Phi$-optimal for estimating $Q^T\tau$ if and only if its treatment proportions design is $\Phi$-optimal in (\ref{eModelWithoutTrend}) and $\xi$ is balanced.
\end{corollary}

\subsection{Rank Deficient Subsystems}\label{sRankDeficient}

Until now, we always assumed that the $v \times s$ matrix $Q$ has full rank. However, there are some frequently used sets of contrasts that do not satisfy this assumption. Such subsystems of interest are called rank deficient subsystems; for a detailed study of such systems, see \cite{puk}. An example of a rank deficient subsystem is the set of contrasts determined by the matrix $Q=I_v-\frac{1}{v}J_v$ which aims at estimating the centered effects of treatments (see \cite{Pukelsheim83}).

In the rank deficient subsystems, the information matrix $N_K(\xi)$ is not well defined. Instead, following \cite{puk}, for a feasible design we define the matrix $C_K(\xi):=(K^T M^-(\xi)K)^+$, where the superscript $+$ denotes the Moore-Penrose inverse. For $K=\big(Q^T, 0\big)^T$, we get $C_K(\xi) = (Q^T M_\tau^-(\xi)Q)^+$.
Then, if $\Phi(N)$ depends only on the eigenvalues of $N$, instead of maximizing $\Phi$ defined on all eigenvalues of $N_K(\xi)$, we maximize $\Phi$ defined on the positive eigenvalues of $C_K(\xi)$.  

For the full rank subsystem, the eigenvalues of the information matrix $N_K(\xi)$  are the inverses of the eigenvalues of $K^TM^-(\xi)K$. Similarly, the matrix $C_K(\xi)$ satisfies that its non-zero eigenvalues are inverses of the non-zero eigenvalues of the matrix $K^TM^-(\xi)K$. Thus, at least in the sense of their spectra, the matrices $C_K(\xi)$ are an analogue to the information matrices for full rank subsystems.


In the rank deficient case, results analogous to the full rank case hold.
We will show that by introducing the nuisance effects, we cannot increase information about the treatment contrasts, as measured by $C_K(\xi)$.
The ordering of matrices $C_K(\xi)$ is induced by the inverse ordering of the matrices $K^T M^-(\xi) K$. For any design $\xi$, we obtain $K^T M^-(\xi) K = Q^TM_\tau^-(\xi)Q$ and for its treatment proportions design $Q^T M^-(w) Q = Q^TM_{11}^-(\xi)Q$. Moreover, $M_\tau(\xi) = M_{11}(\xi) - M_{12}(\xi) M_{22}^-(\xi) M_{12}^T(\xi) \preceq M_{11}(\xi) $, therefore there exist generalized inverses that satisfy $M_\tau^-(\xi) \succeq M_{11}^-(\xi) $ (see \cite{Wu}) and it follows that $K^T M^-(\xi) K \succeq Q^TM^-(w)Q$. As $\Phi(N)$ depends only on the eigenvalues of $N$ and the Moore-Penrose inverse $X^+$ has inverse non-zero eigenvalues of $X$, it implies that $\Phi\big((K^T M^-(\xi) K)^+\big) \geq \Phi\big((Q^TM^-(w)Q)^+\big)$, i.e., $\Phi(C_K(\xi)) \geq  \Phi(C_Q(w))$.
\\

From part (ii) of Proposition \ref{pBalanceSameIM} it follows that any nuisance resistant design $\xi$ has the same matrix $C_K(\xi)$ as its treatment proportions design, i.e., $C_K(\xi) = C_Q(w)$. Hence, Theorems \ref{tOptWeights} and \ref{tBalanceOpt} hold even in the rank deficient case. 

\begin{theorem} \label{tRankDefMain}	
	Let $\Phi$ be an information function and let $Q$ be a $v \times s$ matrix of contrasts with $\mathrm{rank}(Q) < s$. Let $w^*$ be a $\Phi$-optimal design for estimating $Q^T \tau$ in model (\ref{eModelWithoutTrend}). Then, the following holds
	\begin{enumerate}
		\item[(i)] Any nuisance resistant design $\xi$, whose treatment proportions design is $w^*$, is $\Phi$-optimal for estimating $Q^T\tau$ and $C_K(\xi^*) = \big(Q^T \diag((w^*)^{-1}) Q\big)^+$.
		\item[(ii)] If $\xi^*$ is a $\Phi$-optimal design in model (\ref{eModel1}) and $w^*$ is its treatment proportions design, then $w^*$ is $\Phi$-optimal for estimating $Q^T\tau$ in model (\ref{eModelWithoutTrend}).
	\end{enumerate}
\end{theorem}

In particular, we obtain optimality of balanced and product designs with optimal treatment weights.

\subsection{Optimal treatment proportions for selected systems of contrasts}\label{ssContrasts}

We say that the system of contrasts $Q^T\tau$ is \emph{completely symmetric} if $QQ^T$ is completely symmetric. It is easy to show that such $Q$ must satisfy $QQ^T=a(I_v-J_v/v)$ for some $a > 0$. We will show that some common systems of contrasts are completely symmetric.

We consider information functions $\Phi$ that are orthogonally invariant, i.e., $\Phi(UNU^T)=\Phi(N)$ for any orthogonal matrix $U$. Note that a function $\Phi$ is orthogonally invariant if and only if $\Phi(N)$ depends only on the eigenvalues of $N$ (see, e.g., \cite{Harman04} for further details).

\begin{theorem}\label{tCScontrasts}
	Let $Q^T\tau$ be a completely symmetric system of contrasts. Then the uniform treatment proportions design $\bar{w}=1_v/v$ is $\Phi$-optimal for estimating $Q^T\tau$ with respect to any orthogonally invariant information function $\Phi$.
\end{theorem}

By a \emph{maximal system of orthonormal contrasts}, we mean a set of $v-1$ contrasts that are orthogonal to each other and have norm 1, i.e., $q_1, \ldots, q_{v-1}$ satisfy $q^T_i q_j = 0$ for $i \neq j$ and $q_i^Tq_i=1$ for all $i$. Note that a special case of the maximal system of orthonormal contrasts are the Helmert contrasts (see, e.g., \cite{CoxReid}, Appendix C).
Since $Q$ is a $v \times (v-1)$ matrix of orthonormal contrasts, the matrix $[Q, 1_v/\sqrt{v}]$ is orthogonal. It follows that $Q^TQ = I_{v-1}$ and $QQ^T = I_v - J_v/v$, thus it is a completely symmetric system. It is easy to verify that the information matrix of a treatment proportions design $w>0$ is $N_Q(w) = Q^TM(w)Q - Q^TM(w)J_vM(w)Q$ and in particular $N_Q(\bar{w})=v^{-1}I_{v-1}$.

Consider a system of centered treatment effects, or \emph{centered contrasts}, which is the system of contrasts $\tau_1 - \bar{\tau}, \ldots, \tau_v - \bar{\tau}$, where $\bar{\tau}$ is the mean of the treatment effects. That is, $Q=I_v-J_v/v$, which is a $v \times v$ matrix of rank $v-1$ and thus $Q^T\tau$ is a rank deficient system. Note that $QQ^T=Q$ is completely symmetric and hence the centered contrasts are a completely symmetric system of contrasts.
In Section 5 of the paper \cite{Pukelsheim83}, this system of contrasts was analyzed in great detail for a special case of model (\ref{eModel1}), the block designs, and the optimality of product designs with uniform treatment weights was obtained. The matrix $C_Q(\bar{w})$ of the uniform treatment design satisfies $C_Q(\bar{w}) = (vQ^TQ)^+ = v^{-1}Q = I_v/v - J_v/v^2$.

By a \emph{system of all pairwise comparisons} we mean the system of $\tau_i-\tau_j$ for all $i>j$ (considered in, e.g., \cite{BaileyCameron}). The corresponding $v \times \frac{v(v-1)}{2}$ matrix $Q$ satisfies $QQ^T=vI_v-J_v$ and thus the system is completely symmetric.

\begin{corollary}
	The uniform treatment design $\bar{w}$ is $\Phi$-optimal for estimating the system of orthonormal contrasts, the system of centered contrasts as well as the system of all pairwise comparisons, with respect to any orthogonally invariant information function $\Phi$.
\end{corollary}

Consider an experiment which aims at comparing two sets of treatments. Exact designs for these contrasts were studied in multiple design settings, e.g. in \cite{Majumdar86} and \cite{Jacroux02} in block experiments, \cite{Jacroux90} studied $A$- and $MV$-optimal designs in model (\ref{eModelWithoutTrend}), \cite{Jacroux93} and \cite{GithinjiJacroux} considered the presence of trends. Without loss of generality, let the first set consist of the first $g$ (control) treatments, $0<g<v/2$,  and the second set be the set of the remaining $v-g$ treatments. Then the aim is to estimate all treatment comparisons $\tau_j - \tau_i$, where $1\leq i \leq g$ and $g+1 \leq j \leq v$, which leads to matrix $Q=(-I_g \otimes 1_{v-g}, 1_{g} \otimes I_{v-g})^T$, where $\otimes$ denotes the Kronecker product. In \cite{Majumdar86} the author suggests that such situation may arise when comparing two 'packages' of treatments or in comparing a set of new treatments with a set of standard (control) treatments. We will call such system of treatment contrasts \emph{comparison of treatments with controls}. This system of contrasts naturally generalizes the standard system for comparison of $v-1$ treatments with one control, $\tau_2-\tau_1$, \ldots, $\tau_v - \tau_1$, where $g=1$.

\begin{theorem}\label{tCwControls}
	Let $p \in [-\infty,0]$. If $p>-\infty$, let $\gamma_p$ be the unique solution of the equation
	\begin{equation}\label{eCwControls}
		(v-g-1)\gamma^{1-p} - (g-1)(1-\gamma)^{1-p} + 2\gamma - 1 = 0
	\end{equation}
	in the interval $(0,1/2]$ and let $\gamma_{-\infty}=1/2$. Then the treatment proportions design that satisfies $w_1 = \ldots = w_g = \gamma_p/g$ and $w_{g+1} =\ldots = w_v = (1-\gamma_p)/(v-g)$ is $\Phi_p$-optimal for comparison of $(v-g)$ treatments with $g$ controls, $0<g<v/2$.
\end{theorem}

We note that for any $p \in (-\infty,0]$ and $v > 2$ there exists a unique solution $\gamma_p$ of the equation (\ref{eCwControls}) in the interval $(0,1/2)$, which is moreover numerically easy to calculate, because the function $F(\gamma) = (v-g-1)\gamma^{1-p} - (g-1)(1-\gamma)^{1-p} + 2\gamma - 1$ is an increasing convex function for $\gamma \in (0,1/2)$ that satisfies $F(0)\leq 0$ and $F(1/2)\geq 0$.

The obtained optimal treatment proportions which depend on the choice of criterion $\Phi_p$ generalize the results obtained for block designs by \cite{GiovagnoliWynn} on comparison with (one) control. The optimal proportions given by Theorem \ref{tCwControls} are characterized by a single value, $\gamma_p$, the total weight of the first $g$ treatments. In particular, for $D$-optimality, $\gamma_0 =g/v$ and the optimal proportions are uniform; for $A$-optimality, $\gamma_{-1} = \frac{\sqrt{g(v-g)}-g}{v-2g}$ which lies in $(g/v,1/2)$; and for $E$-optimality, $\gamma_{-\infty}=1/2$, i.e., to each of the two sets of treatments, half of the total weight is allocated.
\bigskip

For comparison with controls, it is common to also use the criterion of $MV$-optimality which minimizes the maximum variance of the contrasts of interest. It turns out that $MV$-optimal and $A$-optimal treatment proportions are the same. It
follows that the $A$- and $MV$-optimal nuisance resistant designs are the same.

\begin{theorem}\label{tMVopt}
	Let $\gamma = \frac{\sqrt{g(v-g)}-g}{v-2g}$ and
	let $w_1 = \ldots = w_g =\gamma/g$ and $w_{g+1} = \ldots = w_v = (1-\gamma)/(v-g)$. Then $w$ is $MV$-optimal for comparison of $(v-g)$ treatments with $g$ controls, $0<g<v/2$.
\end{theorem}

Once optimal treatment proportions are calculated, optimal nuisance resistant (balanced, product) designs can be constructed, retaining the same information matrix as their treatment proportion designs.
Note that the matrix $Q$ for completely symmetric contrasts and for comparison with controls has rank $v-1$, therefore for such systems, the balanced designs and nuisance resistant designs coincide.

\section{Examples}\label{sExamples}

\subsection{Trend Resistant Designs}

Let us consider a model where we perform the trials in a time sequence, in each time exactly one trial, and the nuisance effect is the effect of some time trend
\begin{equation}
	\label{eTimeTrendModel}
	Y_i = \tau_{u(i)} + h_1(t(i))\theta_1 + \ldots + h_d(t(i))\theta_{d} + \varepsilon_i, i \in \{1,\ldots,n\},
\end{equation}
where $u(i) \in \{1,\ldots,v\}$ represents the chosen treatment and $t(i) \in \{1,\ldots,n\}$ denotes in which time the treatment is to be applied in trial $i$. The functions $h_1, \ldots, h_d: \mathbb{R}\rightarrow\mathbb{R}$ are the regressors of the time trend, often chosen to be polynomials of degrees $0, \ldots, d-1$ respectively.

The interest in designs that perform well under model (\ref{eTimeTrendModel}) dates back to the mid-20th century, e.g., in paper \cite{Cox}. The research focus is usually on combinatorial construction of exact designs \emph{orthogonal} to time trend (or \emph{trend free}). These are designs that satisfy that no information is lost due to the time trend (see, e.g., \cite{JacrouxMajumdar}, \cite{Bailey}). Usually, the focus is on all parameters of interest, not on a system of contrasts $Q$, resulting in the condition that a design is trend free with respect to $h_k$ in model (\ref{eTimeTrendModel}) if $\sum_t \xi(u,t) h_k(t) = 0$, see, e.g. \cite{Cox}. Such trend free designs satisfy $M_{12}(\xi)=0$ and thus $M_\tau(\xi) = M_{11}(\xi)$.

The drawback of the combinatorial approach is that it is usually tailored for a very specific model. For example, the theoretical results on orthogonal designs require the number of design points to be a multiple of the number of treatments, the time points to be evenly spaced and the time trend needs to be represented by a polynomial. However, these conditions often do not hold.
The reader may find a survey of the literature on the trend resistant experimental designs in the papers \cite{Cheng} or \cite{AtkinsonDonev}.

Note that the orthogonal designs satisfy (\ref{eBalanceCondGeneral2}) and thus they are balanced. However, since we aim at estimating a set of treatment contrasts $Q$, the stringent conditions of orthogonality need not hold for the information to be retained. If $\xi$ is resistant to nuisance effects, the equality $M_\tau(\xi) = M_{11}(\xi)$ in general does not hold, but such $\xi$ satisfies $N_K(\xi) = N_Q(w)$, i.e., the designs resistant to nuisance effects eliminate the effects of the time trend. We remark that when $\sum_t h(t) = 0$, the conditions of orthogonality and the conditions of balancedness coincide.
\bigskip

We will examine the model with trigonometric time trend of degree $D \in \mathbb{N}$, which can be used to model, for instance, circadian rhythms (cf. \cite{Kitsos}). For simplicity, let $\phi_n=2\pi/n$ and consider the model
\begin{align}\label{eTrigonometricModel}
	Y_t=\tau_{u(t)}&+\theta_0+\theta_1\cos(\phi_nt)+\theta_2\sin(\phi_nt)+\ldots  \\
	&+\theta_{2D-1}\cos(D\phi_nt)+\theta_{2D}\sin(D\phi_nt)+\varepsilon_t, \nonumber
\end{align}
where $t = 1, 2, \ldots, n$.

An exact design $\xi$ will be represented by a sequence of treatments determining which treatments are to be chosen in which times. Note that the regression functions satisfy $\sum_t h_k(t)=0$ for $k>0$, i.e., the notions of orthogonal and balanced designs for this model coincide.

Using Theorem \ref{tBalanceOpt}, we get that by repeating a sequence of treatments with $\Phi$-optimal treatment weights, we may obtain a $\Phi$-optimal design for model (\ref{eTrigonometricModel}) of high degree.

\begin{proposition}\label{pTrig}
	Let $\Phi$ be an information function. Let $l \in \mathbb{N}$ and let $\xi_p$ be an exact design of size $l$ with $\Phi$-optimal treatment proportions for estimating contrasts $Q^T\tau$. Let $m \in \mathbb{N}$. Then, the exact design $\xi=\xi_p\xi_p...\xi_p$ of size $n=lm$ formed by an $m$-fold replication of $\xi_p$ is $\Phi$-optimal for all trigonometric models \eqref{eTrigonometricModel} of degrees $D<m$.
\end{proposition}

It is in fact possible to show that the design $\xi$ from Proposition \ref{pTrig} is $\Phi$-optimal for models of the type \eqref{eTrigonometricModel} of any degree, but they cannot include the terms $\cos(a\phi_nt)$ and $\sin(a\phi_nt)$, where $a$ is an integer multiple of $m$. 

We demonstrate the results given by Proposition \ref{pTrig} on a simple example.

\begin{example}
	Consider the model 
	$$Y_{t} = \tau_t + \theta_0 + \theta_1 \sin(r) + \theta_2 \cos(r) + \theta_3 \sin(2r) + \theta_4 \cos(2r) + \theta_5 \sin(3r) + \theta_6 \cos(3r) + \varepsilon_{t},$$
	where $r=\frac{2\pi}{n}t$ and $t = 1, 2, \ldots, n$. Let $n=16$, $v=3$ and $\xi_p = 2113$. Then, the design $\xi_1^* = 2113,2113,2113,2113$ is $E$-optimal for comparison with one control. Let $n=12$, $v=3$ and $\xi_q=321$. Then, the design $\xi_2^* = 321,321,321,321$ is $D$-optimal for comparison with one control. Moreover, let $\Phi$ be an orthogonally invariant information function. Then, $\xi_2^*$ is $\Phi$-optimal for estimating any completely symmetric system of contrasts. \qed
\end{example}

\subsection{Block Designs, Row-Column Designs}

Consider an experiment, where the treatment units are arranged in $b$ blocks. As usual, for each of the $N$ treatment units, we choose one of $v$ treatments. The response is then determined by the treatment effects and block effects. We assume that the treatment and block effects do not interact, i.e, we obtain an additive blocking experiment
\begin{equation}
	\label{eBlockModel}
	Y_i = \tau_{u(i)} + \eta_{t(i)} + \varepsilon_i,\quad i = 1,\ldots,N,
\end{equation}
where $u(i) \in \{1,\ldots,v\}$ and $t(i) \in \{1,\ldots,b\}$. The designs of blocking experiments are called \emph{block designs}. There is a large amount of literature on this topic, in particular the papers that consider treatment contrasts in block designs are, e.g., \cite{MajumdarNotz}, \cite{Pukelsheim83}.

Note that model (\ref{eBlockModel}) may be expressed as a special case of model (\ref{eModel1}), where $\T=\{1,\ldots,b\}$, $n=b$, $\theta = (\eta_1, \ldots, \eta_b)^T$ and $h(t)=e_t \in \mathbb{R}^{b}$ is the $t-$th elementary unit vector.
For block designs, in the balance conditions (\ref{eBalanceCondGeneral2}) we obtain $\xi(1,t)/w_1 = \ldots = \xi(v,t)/w_v$ for all $t \in \{1,\ldots,b \}$, which leads to a product design $\xi = w \otimes \alpha$. That is, all balanced designs in model (\ref{eBlockModel}) are product designs. Therefore, for a system of contrasts of rank $v-1$ and a strictly concave information function $\Phi$, from Theorem \ref{tSuffAndNecc} it follows that all $\Phi$-optimal designs are product designs. Note that, in general, the balanced incomplete block designs and the balanced treatment incomplete block designs (see, e.g., \cite{MajumdarNotz}) are not balanced in the sense of conditions \ref{eBalanceCondGeneral2}.
\bigskip

Block designs are often used for eliminating heterogeneity in one direction, e.g., caused by a nuisance time trend. If the position of a unit within a block affects the response as well, or in general, the heterogeneity needs to be eliminated in two directions, we may use the \emph{row-column designs} (see \cite{Jacroux}). Here, $N$ experimental units are arranged in $b_1$ rows and $b_2$ columns. The mean response is determined by the sum of the treatment, row and column effect, modelled as
\begin{equation}
	\label{eRowColumnModel}
	Y_i = \tau_{u(i)} + \eta_{k(i)} + \varphi_{l(i)} + \varepsilon_i,\, i = 1,\ldots,N,
\end{equation}
where $u(i) \in \{1,\ldots,v\}$, $k(i) \in \{1,\ldots,b_1\}$ and $l(i) \in \{1,\ldots,b_2\}$ represent the row and column chosen for the $i$-th trial, respectively, and $\eta_{k(i)}, \varphi_{l(i)}$ are the row and column effects.

This model can also be expressed as a special case of model (\ref{eModel1}), where $\T=\{1,\ldots,b_1\}\times \{1,\ldots,b_2\}$, $n=b_1b_2$, $\theta = (\eta_1, \ldots, \eta_{b_1}, \phi_1, \ldots, \phi_{b_2})^T$ and $h(k,l)=(e^T_k, e^T_l)^T \in \mathbb{R}^{b_1 + b_2}$. The balance conditions for the row-column model become $w_1^{-1}\sum_l \xi(1,k,l) = \ldots = w_v^{-1}\sum_l \xi(v,k,l)$ for all $k=1,\ldots, b_1$ and $w_1^{-1}\sum_k \xi(1,k,l) = \ldots = w_v^{-1}\sum_k \xi(v,k,l)$ for all $l=1,\ldots, b_2$. That is, for any row (column) the ratio of the total weights of any two treatments $i$, $j$ in the particular row (column) is given by the ratio of the treatment weights $w_i/w_j$. In other words, for the design $\xi$ to be balanced (and hence optimal, if $\xi$ attains optimal treatment weights), the functions $\xi(u,\cdot,\cdot)/w_u$ need to have the same row and column marginals for all $u=1,\ldots, v$.

The block and row-column designs are called the designs for the one-way and two-way elimination of heterogeneity, respectively (see \cite{Jacroux}). By combining the models (\ref{eBlockModel}) and (\ref{eTimeTrendModel}), the blocking experiment under the presence of a nuisance time trend is obtained, see, e.g., \cite{BradleyYeh} or \cite{JacrouxMajumdar}, which we will examine further in Example \ref{exBlockTrend}.

\section{Constructing Efficient Exact Designs}\label{sConstructing}

By constructing product designs with optimal treatment weights, and calculating their criterial values (or by analytically deriving optimal criterial values), we may assess the quality of the exact designs. More precisely, we can compute lower bounds on the efficiency of any given exact design by calculating its approximate efficiency with respect to the criterion $\Phi$, $\mathrm{eff}(\xi)=\frac{\Phi(\xi)}{\Phi(\xi^*)}$, where $\xi^*$ is a $\Phi$-optimal approximate design.
Moreover, as we demonstrate in this section, the balance conditions provide a tool for obtaining optimal approximate designs with small support and these designs can be used to construct efficient exact designs.

We will focus on exact designs of experiments in which exactly one trial is to be performed under each nuisance condition. The problem of finding such optimal designs is in general a difficult discrete optimization problem, see, e.g., \cite{AtkinsonDonev} or \cite{HarmanSagnol}.
\bigskip

Note that both the balance conditions (and, in general, the conditions of resistance to nuisance effects) and the conditions on $\Phi$-optimal weights are linear. Hence, results provided in the previous sections can be used to calculate a balanced approximate design with $\Phi$-optimal weights employing linear programming, solving the problem
\begin{equation}\label{eLP}
	\min \{c^T x | Ax = b, x \geq 0 \},
\end{equation}
where $x \in \mathbb{R}^{vn}$ represents a design $\xi$ in the vector form, $A$ consists of sufficient conditions of optimality and we are free to choose the the vector c of the coefficients of the objective function. Let us denote the set of all feasible solutions of (\ref{eLP}) as $\mathcal{P}$. 

The matrix $A$ consists of 
\begin{enumerate}
	\item[(i)] $v$ equalities $\sum_t \xi(u,t) = w_u$, $u=1,\ldots,v$, i.e., $\xi$ attains the $\Phi$-optimal treatment weights,
	\item[(ii)] $d(v-1)$ equalities $w_1^{-1} \sum_t \xi(1,t) h(t) = w_u^{-1} \sum_t \xi(u,t) h(t)$, $u=2,\ldots, v$, i.e., $\xi$ is a balanced design,
	\item[(iii)] $n$ equalities $\sum_u \xi(u,t)=1/n$, i.e., under each nuisance condition exactly one trial is performed.
\end{enumerate}

Once the $\Phi$-optimal treatment weights $w^*$ are obtained, a $\Phi$-optimal design can be constructed as a product $w^* \otimes \alpha$ for any nuisance conditions design $\alpha$. However, in general, it is difficult to construct exact designs from the product designs, due to their regular structure and large support.
To obtain an optimal design with small support, it is beneficial to employ the simplex method of linear programming, whose output is an optimal design $\xi^*$ that represents a vertex in $\mathcal{P}$, the set of feasible solutions of (\ref{eLP}).

\begin{proposition}\label{pSimplex}
	Let $\xi$ represent a vertex in $\mathcal{P}$. Then, $\xi$ contains at most $v+(v-1)k +n-1$ support points, where $k$ is the affine dimension of the set $\{h(t)\}_{t \in \T}$.
\end{proposition}

From Proposition \ref{pSimplex} it follows that by employing the simplex method, we can obtain a $\Phi$-optimal design $\xi^*$ with at most  $(v-1)(k+1)+n$ support points.
As a special case, when a constant term $\theta_0$ is present in the time trend, it may be ignored in the conditions in (\ref{eLP}), because it does not increase the affine dimension of $\{h(t)\}_{t \in \T}$; reducing thus the upper bound on the number of support points by $v-1$.

Suppose that $\xi$ satisfying (iii) has support of size $n$, the number of nuisance conditions. Then, $\xi$ uniquely determines an exact design of size $n$. The number of support points in designs obtained by the simplex method is only slightly larger than $n$; it exceeds this minimum support size by $(v-1)(k+1)$. Note that the number of exceeding support points does not depend on $n$, thus, even for increasing number of nuisance conditions, it remains small.

We note that using the Carathéodory Theorem (cf. Theorem 8.2. in \cite{puk}), it is possible to obtain results similar to Proposition \ref{pSimplex}, but the Carathéodory Theorem does not provide an actual method of constructing a design with small support, unlike the simplex method.
\bigskip

From an optimal approximate design with small support, an efficient exact design can be constructed by rounding, or often even by a complete enumeration of treatments in a small number of nuisance conditions.

\begin{example}\label{exTrendLinprog}
	Consider an experiment of performing trials in a time sequence
	$$ Y_i = \tau_{u(i)} + \theta_0 + \theta_1h_1(t(i)) + \varepsilon_i,\quad i=1,\ldots,n, $$
	where $h_1(t)=e^t/\sum_j e^j$ represents an exponential time trend (e.g., the decay of wool in the experiment of wool processing, as suggested by \cite{AtkinsonDonev}).
	Assume that $v=5$ and the objective is to find an $A$-optimal design for comparing 3 treatments with 2 controls, i.e., $\gamma_{-1} = \sqrt{6}-2 \approx 0.45$. We will provide optimal balanced designs with small support by employing the simplex method of linear programming (\emph{linprog} function of Matlab, using the simplex algorithm). Note that also the interior point (default) algorithm of Matlab's \emph{linprog} tends to provide optimal vertex solutions or optimal solutions with small support; as such it can be used instead of the simplex algorithm.
	
	First, let $n=8$. Since $\theta_0$ is the constant term, from Proposition \ref{pSimplex} it follows that there are at most $n-1 + v + (v-1)D = 16$ linearly independent rows of $A$ in (\ref{eLP}), where $D=d-1=1$.
	We remark that although we are free to choose the vector $c$ in the linear program, the support size of the design obtained by the simplex method does not seem to depend on the choice of $c$. Therefore, we chose each of the elements of $c$ uniformly randomly from $(0,1)$.
	
	We obtained a design $\xi_1$ that has the support of size $16$ (and the minimum support size is 8) and is ``fixed'' in 4 times (i.e., in each of these times $\xi_1$ has only one non-zero element), see Table \ref{tblLPdesign}. The support size corresponds to the bound $16$ given by \ref{pSimplex}.
	
	\begin{table}[h]
	\begin{center}
	\begin{tabular}{  c | c c c c c c c c }
	\hline\noalign{\smallskip}
	$u\backslash t$ & 1 & 2 & 3 & 4 & 5 & 6 & 7 & 8 \\
	\noalign{\smallskip}\hline\noalign{\smallskip}
	1 &         0 &   0.1250  &  0.0560     &    0    &     0     &    0   &      0  &  0.0437 \\
	2 & 0   &      0  &  0.0690   &      0   &      0 &   0.1250 &   0.0059 &   0.0249 \\
	3 & 0.0245     &    0      &   0  &  0.1250   &      0     &    0   &      0 &   0.0340 \\
	4 & 0.0154     &    0     &    0   &      0   & 0.1250     &    0  &  0.0207  &  0.0224 \\
	5 & 0.0851     &    0     &    0    &     0   &      0     &    0  &  0.0984  &       0 \\
	\noalign{\smallskip}\hline
	\end{tabular}
	\caption{$A$-optimal balanced approximate design obtained by the simplex method. The first two treatments are controls.}
	\label{tblLPdesign}
	\end{center}
	\end{table}
	
	By a complete enumeration of the possible treatment combinations in the remaining 4 non-fixed times, we chose the design $\hat{\xi}_1=  51234215$ that maximizes the criterial value. 
	For $v=5$ and $n=8$ it is possible to find the $A$-optimal exact design by a complete enumeration, $\xi^*=41253214$. It turns out that the design $\hat{\xi}_1$ has efficiency 1 relative to $\xi^*$, i.e., it is optimal; in fact, $\hat{\xi}_1$ can be obtained by relabelling treatments $3,4,5$ in $\xi^*$. Note that, in general, the proposed heuristic does not provide optimal exact designs.
	
	For $n=100$ and the same model assumptions, we obtained a design $\xi_2$ with support of size $108$. That is, the number of support points of $\xi_2$ exceeds the minimum support size again by 8; moreover $\xi_2$ has only 5 non-fixed times. 
	Therefore, even for $n=100$, an efficient exact design may be constructed by a complete enumeration of treatments in the non-fixed times. The resulting design $\hat{\xi}_2$ assigns 23, 22, 19, 18, 18 trials to treatments $1, \ldots, 5$, respectively, which corresponds to the $A$-optimal treatment weights given by $\gamma_{-1}\approx0.45$. Moreover,  $\hat{\xi}_2$ has approximate efficiency 0.994; its efficiency relative to the optimal exact design would be even higher, but for the problem of this size, it is infeasible to compute an optimal exact design by a complete enumeration.
	\qed
	
\end{example}

In the following example, we demonstrate for various values of $v$, $n$, $d$ that the simplex method provides optimal approximate designs with small support.

\begin{example}\label{exTrendLPmany}
	Consider an experiment of performing trials in a time sequence which aims at comparing treatments with control and the observed values are a subject to a polynomial time trend
	$$ Y_i = \tau_{u(i)} + \theta_0 p_0(t(i)) + \theta_1p_1(t(i)) + \ldots + \theta_D p_D(t(i)) + \varepsilon_i,\quad i=1,\ldots,n, $$
	where $p_0, \ldots, p_D$ are discrete orthogonal polynomials of degrees $0, \ldots, D$, respectively, i.e., $\sum_t p_i(t) p_j(t) = 0$ for $i \neq j$. Furthermore, we set $p_0 \equiv 1$ and $p_i(1) = 1$ for all $i$. Note that although the total number of time trend parameters is $d=D+1$, the term $\theta_0 p_0(t) = \theta_0$ represents the constant term and thus, from Proposition \ref{pSimplex} it follows that there are at most $n+(v-1)(D+1)$ linearly independent rows of $A$ in (\ref{eLP}).
	
	For varying $v$, $n$ and $D$, we calculated an $A$-optimal design for comparison of treatments with one control using the simplex method and we compared the size of its support with the minimum size of the support and with the theoretically derived bounds given by Proposition \ref{pSimplex} (see Table \ref{tblLPdesigns}).
	
	\begin{table}[h]
		\begin{center}
		\begin{tabular}{  c c c | c c }
			\hline\noalign{\smallskip}
			$v$ & $n$ & $D$ & Simplex & Max. Simplex \\
			\noalign{\smallskip}\hline\noalign{\smallskip}
			3 & 120 & 1 & 124 (4) & 124 \\
			3 & 150 & 1 & 154 (4) & 154 \\
			3 & 200 & 1 & 204 (4) & 204 \\
			4 & 120 & 1 & 126 (6) & 126 \\
			5 & 120 & 1 & 128 (8) & 128 \\
			8 & 120 & 1 & 134 (14) & 134 \\
			3 & 120 & 2 & 126 (6) & 126 \\
			3 & 120 & 3 & 128 (8) & 128 \\
			3 & 120 & 4 & 130 (10) & 130 \\
			3 & 120 & 5 & 132 (12) & 132 \\
			\noalign{\smallskip}\hline
		\end{tabular}
		\caption{The size of the support. For a given number of treatments $v$, number of times $n$ and degree of the time trend $D$, the column \emph{Simplex} contains the size of the support of the design calculated using the simplex method (and the number of support points over the minimum size of the support, $n$, in parentheses); the column \emph{Max. Simplex} contains the theoretical bound on the maximum number of support points given by Proposition \ref{pSimplex}. Note that for each of the studied cases the theoretical bound on the support has been exactly achieved.}
		\label{tblLPdesigns}
		\end{center}
	\end{table}
	\qed
	
\end{example}

\begin{example}\label{exBlockTrend}
	Consider the model given by \cite{BradleyYeh}. We have a blocking experiment of $b$ blocks, each of size $l$, where the response of a trial is also influenced by a common trend effect determined by the position of the unit within the block. In each block, there is exactly one trial performed on each position. Moreover, the trend effect in position $t_2(i)$ does not depend on the particular block $t_1(i)$. We have
	\begin{equation}\label{eBlockTrend}
		Y_i = \tau_{u(i)} + \eta_{t_1(i)} + p^T(t_2(i))\varphi + \varepsilon_i,\quad i=1,\ldots,n,
	\end{equation}
	where $t_1(i) \in \{1,\ldots,b\}$ is the block in which trial $i$ is performed, $\eta_{t_1}$ is the effect of the $t_1$-th block, $t_2 \in \{1,\ldots, l\}$ denotes the position of the unit within the block, $n=bl$, $\varphi$ is a $(D+1) \times 1$ vector of nuisance trend effects and $p:\mathbb{R} \rightarrow \mathbb{R}^{D+1}$ is a regression function of the nuisance trend.
	
	Assume that $v=3$, $b=3$ and $l=8$, $n=3\times 8=24$ and that the time trend is modelled by discrete orthogonal polynomials $p_0, p_1, p_2$ of degrees $0,1,2$, i.e., $D=2$. We aim to find an $E$-optimal design for comparing treatments with one control. The optimal weight of the first treatment is $\gamma^*=1/2$ and the optimal weights of the other two are $1/4$.
	
	The conditions (ii) in $A$ can be expressed as two sets of conditions: (ii.a) $(v-1)b$ conditions $w_1^{-1} \sum_{t_2} \xi(1,t_1,t_2) = w_u^{-1} \sum_{t_2} \xi(u,t_1,t_2)$ for $u=2,\ldots, v$ and $t_1=1,\ldots, b$, and (ii.b) $(v-1)(D+1)$ conditions $w_1^{-1}\sum_{t_1,t_2}\xi(1,t_1,t_2)p(t_2) = w_u^{-1}\sum_{t_1,t_2}\xi(u,t_1,t_2)p(t_2)$, $u=2,\ldots, v$. By summing (ii.a) over all $t_1$, and using the fact that $\sum_t \xi(u,t)=1$, we obtain (i), which reduces the number of linearly independent rows in $A$ by $v$.
	Similarly to Proposition \ref{pSimplex}, using (i), it follows that there are at most $(v-1)(b+r) +n-1$ linearly independent rows in $A$, where $r$ is the affine dimension of the set $\{p(t_2)\}_{t_2}$. Since $p_0\varphi_0$ represents the constant term, the number of linearly independent rows in $A$ is at most $(v-1)(b+D) +n-1 = 33$. The minimum number of support points is $n=24$.
	
	Using the simplex method, we obtained an $E$-optimal balanced approximate design $\xi^*$, see Table \ref{tblBlockLPdesign}. The design $\xi^*$ has the support of size 30, which exceeds the minimum support size by 6, and it is fixed in 18 out of the 24 positions.

	\begin{table}[h]
		\begin{center}
			\begin{tabular}{ c| c | c c c c c c c c }
				\hline\noalign{\smallskip}
				block & $u\backslash t$ & 1 & 2 & 3 & 4 & 5 & 6 & 7 & 8 \\
				\noalign{\smallskip}\hline\noalign{\smallskip}
				1 & 1 &         0.0417    &     0  &  0.0417    &     0   & 0.0417  &       0  &  0.0417      &   0   \\    
				&  2 &          0  &  0.0417    &     0    &     0    &     0   & 0.0417     &    0     &    0      \\ 
				& 3 &           0    &     0     &    0  &  0.0417  &       0       &  0    &     0   & 0.0417 \\ \hline
				2 & 1 &         0.0417   & 0.0417    &     0   & 0.0324  &   0 &   0.0417   & 0.0093   &      0 \\
				& 2  &          0    &     0     &    0     &    0 &   0.0417   &      0     &    0  &  0.0417  \\
				& 3  &          0    &     0  &  0.0417 &   0.0093  &     0     &    0   & 0.0324    &     0  \\ \hline
				3 & 1  &        0.0046   &      0  &  0.0083   &      0 &   0.0417  &  0.0417  &  0.0370 &   0.0333  \\
				& 2  &          0   &   0.0417  &  0.0333    &     0    &     0  &       0   &      0  &  0.0083  \\
				& 3 &           0.0370  &       0      &   0 &   0.0417  &       0   &      0  &  0.0046     &    0 \\
				\noalign{\smallskip}\hline
			\end{tabular}
			\caption{$E$-optimal balanced approximate design obtained by a simplex method  for an experiment with 3 blocks, each of size 8, and a common trend effect.}
			\label{tblBlockLPdesign} 
		\end{center}
	\end{table}
	
	By a complete enumeration of treatments in the 8 non-fixed positions, we obtained an exact design $\xi:$ $b_1=12131213, b_2=11312132, b_3=32231111$, where the sequence $b_j$ determines the treatments and their positions in block $j$. Using Theorem \ref{tCwControls}, we get that $\xi$ has approximate efficiency 0.999.
	\qed
\end{example}

\section*{Appendix}


\paragraph{Proof of Proposition \ref{pNuisanceReducesInformation}\\}
	Let us partition the matrix $L$ in
	$$
	N_K(\xi)= \mathrm{min}_{L \in \mathbb{R}^{s \times m}: LK=I_s} LM(\xi)L^T
	$$
	as $L=\big(L_1, L_2\big)$, where $L_1$ is an $s \times v$ and $L_2$ is an $s \times d$ matrix. Then,
	$$\begin{aligned}
	N_K(\xi)
	&=\min_{LK=I_s} LM(\xi)L^T = \min_{(L_1, L_2) (Q^T, 0)^T = I_s} (L_1, L_2)M(\xi)(L_1, L_2)^T \\ 
	&\preceq \min_{L_1Q=I_s} L_1 M_{11}(\xi) L_1^T = N_Q(w).
	\end{aligned}$$
\qed

From now on, we assume that $K^T = \big(Q^T, 0_{s \times d} \big)$.

\begin{lemma}\label{lSameIM}	
	Let $\tilde{M}$ be a non-negative definite matrix. If a design $\xi$ satisfies $M(\xi)\tilde{M}^-K=K$ for some  generalized inverse $\tilde{M}^-$ of $\tilde{M}$,  then 
	(i) $\xi$ is feasible for $K^T\beta$ and
	(ii) $K^T M^-(\xi)K = K^T \tilde{M}^-K$.
\end{lemma}

\begin{proof}
	The steps of the proof follow the proof of Theorem 8.13 from \citet{puk}.
	We denote $G:=\tilde{M}^-$. Since $M(\xi)GK=K$, we obtain $M(\xi)X=K$, where $X=GK$. Therefore $\mathcal{C}(K) \subseteq \mathcal{C}(M(\xi))$ and hence $\xi$ is feasible.
	Let us premultiply the equation $M(\xi)GK=K$ by $K^TM^-(\xi)$ so that we obtain on the right-hand side $K^T M^-(\xi)K$. The left-hand side is then equal to $K^T M^-(\xi)M(\xi)GK$. Note that $K^T=X^TM^T(\xi)=X^TM(\xi)$ and hence the following holds
	$$K^T M^-(\xi)M(\xi)GK = X^TM(\xi) M^-(\xi)M(\xi)GK = X^T M(\xi) G K = K^T G K.$$
	It follows that $K^T M^-(\xi)K = K^T \tilde{M}^- K$.
\end{proof}

\begin{lemma}\label{lMGA}
	Let $w>0$ be a treatment proportions design and let $G:=\diag\big(w^{-1},0_d\big)$. Let $\xi$ be a design in model (\ref{eModel1}), then $\xi$ satisfies $M(\xi)GK=K$ if and only if (i) $w$ is a treatment proportions design of $\xi$ and (ii) $\xi$ is resistant to nuisance effects.
\end{lemma}

\begin{proof}
	We may express $M(\xi)GK = K$ as
	$ M_{11}(\xi)\mathrm{diag}\big(w^{-1}\big) Q = Q $ and $M_{12}^T(\xi) \mathrm{diag}\big(w^{-1}\big) Q = 0$. Since both $M_{11}(\xi)$ and $\mathrm{diag}\big(w^{-1}\big)$ are diagonal matrices, and all rows of $Q$ are assumed to be non-zero vectors, the first equation is equivalent to $\frac{1}{w_u}\sum_t\xi(u,t) = 1$ for all $u$, which is (i).
	From the second equation, we obtain that every row of 
	$
	M_{12}^T(\xi)\mathrm{diag}\big(w^{-1}\big)=
	\begin{bmatrix}
	\frac{1}{w_1}\sum_t \xi(1,t)h(t) & \ldots & \frac{1}{w_v}\sum_t \xi(v,t)h(t)
	\end{bmatrix}
	$ 
	needs to be in $\mathcal{N}(Q^T)$, which is (ii).
\end{proof}


\paragraph{Proof of Proposition \ref{pBalanceSameIM}\\}
	Let $\tilde{M} := \diag(w,0_d)$. Then, $G := \diag(w^{-1},0_d)$ is a generalized inverse of $\tilde{M}$. From Lemma \ref{lMGA} it follows that $M(\xi)GK=K$ and from Lemma \ref{lSameIM} it follows that (i) and (ii) hold. The statement (iii) is a direct consequence of (ii).
\qed



\paragraph{Proof of Theorem \ref{tOptWeights}\\}
	Let $\xi$ be a feasible design in (\ref{eModel1}). Using Proposition \ref{pNuisanceReducesInformation}, we obtain that $N_K(\xi) \preceq N_Q(w)$, where $w$ is the treatment proportions design of $\xi$. 
	Moreover, since product designs are nuisance resistant, part (iii) of Proposition \ref{pBalanceSameIM} implies that  $N_Q(w) = N_K(w \otimes \alpha)$ for any nuisance conditions design $\alpha$. Therefore, $N_K(\xi) \preceq N_K(w \otimes \alpha)$.
	
	Suppose that $w^*$ is not a $\Phi$-optimal design. Then, there exists a design $w_{\mathrm{b}}$ under model (\ref{eModelWithoutTrend}) such that $\Phi(N_Q(w^*)) < \Phi(N_Q(w_{\mathrm{b}}))$. Then, $\Phi(N_K(\xi^*)) \leq \Phi(N_Q(w^*)) < \Phi(N_Q(w_{\mathrm{b}})) = \Phi(N_K(w_{\mathrm{b}} \otimes \alpha))$ for any nuisance conditions design $\alpha$. That is a contradiction with $\xi^*$ being $\Phi$-optimal.
\qed

\begin{lemma}[Theorem 8.13 from \citet{puk}]\label{lAllOptimalDesigns}
	Let $\Phi$ be a strictly concave information function and let $\xi^*$ be $\Phi$-optimal for $K^T\beta$. Let $G$ be a generalized inverse of $M(\xi^*)$ that satisfies the normality inequality of the General Equivalence Theorem (Theorem 7.14 from \citet{puk}), i.e., there exists a non-negative definite matrix $D$ that solves the polarity equation
	$$ \Phi\big(N_K(\xi^*)\big)\Phi^\infty (D) = \mathrm{tr}(CD)=1, $$
	where $\Phi^\infty$ is the polar information function of $\Phi$ (see \citet{puk}),
	and $G$ satisfies the normality inequality
	$$\mathrm{tr}(M(\xi)B) \leq 1 \quad \text{for all feasible designs } \xi, $$
	where $B=GKN_K(\xi^*)DN_K(\xi^*)K^TG^T$.
	Then, a design $\xi$ is $\Phi$-optimal if and only if $M(\xi)GK=K$.
\end{lemma}

In order to use Lemma \ref{lAllOptimalDesigns}, we need to obtain a matrix $G$ that satisfies the normality inequality of the General Equivalence Theorem.

\begin{lemma}\label{lGinvET}
	Let $\Phi$ be a strictly concave information function, let $w^*$ be a $\Phi$-optimal treatment proportions design and let $G:=\diag((w^*)^{-1},0_d)$. Then, $G$ satisfies the normality inequality of the General Equivalence Theorem for estimating $K^T\beta$ in model (\ref{eModel1}). 
\end{lemma}

\begin{proof}
	Let us denote $N^*:=N_Q(w^*)$ and $G_{11}:=\diag\big((w^*)^{-1}\big)$.
	Since $w^*$ is optimal in (\ref{eModelWithoutTrend}), the matrix $G_{11}$ that is the unique generalized inverse of $M(w^*)$, satisfies normality inequality of the General Equivalence Theorem for model (\ref{eModelWithoutTrend}), i.e. there exists a matrix $D$ which satisfies the polarity equation $\Phi(N^*)\Phi^\infty(D)=\mathrm{tr}(N^*D) = 1$ and the matrix  $B_w=G_{11}QN^*DN^*Q^TG_{11}$ satisfies the normality inequality $\mathrm{tr}(M(\tilde{w})B_w) \leq 1$ for all $\tilde{w}$.
	
	There exists a unique $\Phi$-optimal information matrix $N_K(\xi^*)$, because $\Phi$ is strictly concave. Since $N_K(w^* \otimes \alpha) = N_Q(w^*)=N^*$ is $\Phi$-optimal, we have $N_K(\xi^*) = N^*$. Thus, the polarity equality holds in model (\ref{eModel1}) for the same matrix $D$. Let $\tilde{\xi}$ be a feasible design. Then, the left-hand side of the normality inequality in model (\ref{eModel1}) is $\mathrm{tr}(M(\tilde{\xi})B)$, where
	$$
	B=\begin{bmatrix}
	G_{11} & 0 \\ 0 & 0
	\end{bmatrix}
	\begin{bmatrix}
	Q \\ 0
	\end{bmatrix}
	N^* D N^*
	\begin{bmatrix}
	Q^T & 0
	\end{bmatrix}
	\begin{bmatrix}
	G_{11} & 0 \\ 0 & 0
	\end{bmatrix}
	=
	\begin{bmatrix}
	B_w & 0 \\ 0 & 0
	\end{bmatrix}.
	$$
	Then, because $B_w$ satisfies the normality inequality in model (\ref{eModelWithoutTrend}), we obtain $\mathrm{tr}(M(\tilde{\xi})B) = \mathrm{tr}(M_{11}(\tilde{\xi})B_w) = \mathrm{tr}(M(\tilde{w})B_w) \leq 1$, where $\tilde{w}$ is the treatment proportions design of $\tilde{\xi}$.
\end{proof}


\paragraph{Proof of Theorem \ref{tBalanceOpt}\\}
	Let $\xi$ be a design in model (\ref{eModel1}) and $w$ be its treatment proportions design. Since $\Phi$ is isotonic, from Proposition \ref{pNuisanceReducesInformation} it follows that $\Phi(N_Q(w)) \geq \Phi(N_K(\xi))$.
	Since $w^*$ is $\Phi$-optimal, $\Phi(N_Q(w^*)) \geq \Phi(N_Q(w))$ and it is feasible, thus $w^* >0$. Using Proposition \ref{pBalanceSameIM}, we get that $\Phi(N_K(\xi^*)) = \Phi(N_Q(w^*)) \geq \Phi(N_Q(w)) \geq \Phi(N_K(\xi)),$ i.e., $\xi^*$ is $\Phi-$optimal.	
\qed



\paragraph{Proof of Theorem \ref{tSuffAndNecc}\\}
	Let $G=\diag((w^*)^{-1},0_d)$. From Lemma \ref{lGinvET} it follows that $G$ satisfies the normality inequality of the General Equivalence Theorem. Lemma \ref{lAllOptimalDesigns} yields that a design $\xi$ is $\Phi$-optimal if and only if $M(\xi)GK=K$. The equality $M(\xi)GK=K$ holds if and only if $\xi$ satisfies (i) and (ii) from Lemma \ref{lMGA}.
\qed


\paragraph{Proof of Theorem \ref{tCScontrasts}\\}
	
	First, assume that $Q$ has full column rank.
	Let $w$ be a feasible treatment proportions design and let $P$ be a $v \times v$ permutation matrix. We define $Pw$ to be the design given by the $P$-permutation of treatments in $w$, i.e., $Pw(u)=w(\pi_P(u))$  for $u \in \{1,\ldots,v\}$, where $\pi_P$ is the permutation of elements $\{1,\ldots,v\}$ corresponding to the matrix $P$. Since $Pw>0$, it is feasible, its moment matrix is $M(Pw) = PM(w)P^T$ and it has information matrix $N_Q(Pw) = (Q^TPM^{-1}(w)P^TQ)^{-1}$.
	
	We will use the well-known fact that if $X$ is any matrix, the non-zero eigenvalues of the matrices $X^TX$ and $XX^T$ are the same (e.g., 6.54(c) in \citet{Seber}), including multiplicities. Define $Y=Q^TM^{-1/2}(w)$ and $Z=Q^T P M^{-1/2}(w)$. Since $QQ^T$ is completely symmetric, $Y^TY = Z^TZ$.
	Furthermore, $YY^T=Q^T M^{-1}(w) Q = N_Q^{-1}(w)$ and $ZZ^T=Q^T P M^{-1}(w) P^T Q = N_Q^{-1}(Pw)$, thus $N_Q(w)$ and $N_Q(Pw)$ have the same set of non-zero eigenvalues. Since they have the same (full) rank, it follows that $N_Q(w)$ and $N_Q(Pw)$ are orthogonally similar and $\Phi(Pw) = \Phi(w)$.
	Note that analogous results hold in the rank-deficient case for the matrices $C_Q(w)$ and $C_Q(Pw)$.
	
	The uniform treatment design satisfies
	$$\begin{aligned}
	\Phi\big(\bar{w}\big)
	&=
	\Phi\left(\frac{1}{v!}\sum_{P-\text{perm.}} Pw\right) \geq
	\frac{1}{v!}\sum_{P-\text{perm.}} \Phi(Pw)=  \\
	&= \frac{1}{v!}\sum_{P-\text{perm.}} \Phi(w) = \frac{1}{v!}v! \Phi(w) = \Phi(w),
	\end{aligned}$$
	where the inequality follows from the concavity of $\Phi$.
	Thus, $\bar{w}$ is $\Phi$-optimal.
\qed


\paragraph{Proof of Theorem \ref{tCwControls}\\}
	Note that for $Q=(-I_g \otimes 1_{v-g}, 1_{g} \otimes I_{v-g})^T$ we have
	$$
	QQ^T = \begin{bmatrix}
	(v-g)I_g & -J_{g \times (v-g)} \\ -J_{(v-g) \times g} & gI_{v-g}
	\end{bmatrix}.
	$$
	Let $w$ be a treatment proportions design, let $P_1$, $P_2$ be $g \times g$ and $(v-g) \times (v-g)$ permutation matrices, respectively, and let
	\begin{equation}
	\label{ePerm}
	\tilde{P}=\begin{bmatrix}
	P_1 & 0_{g \times (v-g)} \\ 0_{(v-g) \times g} & P_2
	\end{bmatrix}.
	\end{equation}
	Define $\tilde{P}w$ to be the design given by the $\tilde{P}$-permutations of the treatments. 
	Then $M(\tilde{P}w)= \tilde{P}M(w)\tilde{P}^T$ and $\tilde{P}^T QQ^T \tilde{P} = QQ^T$. From an argument analogous to the proof of Theorem \ref{tCScontrasts}, $C_Q(\tilde{P}w)$ and $C_Q(w)$ are orthogonally similar and $\Phi_p(\tilde{P}w) = \Phi_p(w)$.
	
	Define
	$\tilde{w} = \frac{1}{(v-g)!g!} \sum_{\tilde{P}} \tilde{P} w$, where the sum is over all $v \times v$ permutation matrices $\tilde{P}$ of the form (\ref{ePerm}). Then $\Phi_p(\tilde{w}) \geq \Phi_p(w)$. It follows that an optimal design exists in the class of designs that allocate one weight to each of the first $g$ treatments, say $\gamma_1$ ($0<\gamma_1<1/g$), and another weight to each of the other treatments, $\gamma_2:=(1-g\gamma_1)/(v-g)$. Let  $\gamma:=g\gamma_1$ be the total weight of the first $g$ treatments and for a given $\gamma$, we denote such designs as $w_\gamma$.
	
	The non-zero eigenvalues of $C_Q(w_\gamma)$ are inverse to the non-zero eigenvalues of $V(w_\gamma):=Q^TM^{-1}(w_\gamma)Q$, where $M(w_\gamma)=\diag(\gamma_1 1_g, \gamma_2 1_{v-g})$. Let $X=Q^TM^{-1/2}(w_\gamma)$. Then the set of non-zero eigenvalues of $V(w_\gamma) = X^TX$ coincides with the set of non-zero eigenvalues of $$
	XX^T = M^{-1/2}(w_\gamma)QQ^T M^{-1/2}(w_\gamma) =
	\begin{bmatrix}
	(v-g)\gamma_1^{-1}I_g & - (\gamma_1\gamma_2)^{-1/2}J_{g \times (v-g)} \\
	- (\gamma_1\gamma_2)^{-1/2}J_{(v-g) \times g} & g\gamma_2^{-1}I_{v-g}
	\end{bmatrix}.
	$$
	It can be seen that $XX^T$ has the following eigenvalues, listed with the corresponding eigenvectors $x=(x_1^T,x_2^T)^T$, where $x_1 \in \mathbb{R}^{g}$ and $x_2 \in \mathbb{R}^{v-g}$: 
	$\mu_1 = g\gamma_2^{-1}$ with multiplicity (w.m.) $v-g-1$, $x_1=0_g$ and $1_{v-g}^Tx_2=0$;
	$\mu_2 = (v-g)\gamma_1^{-1}$ w.m. $g-1$, $1_g^Tx_1 = 0$ and $x_2=0_{v-g}$;
	$\mu_3 = (v-g)\gamma_1^{-1} + g\gamma_2^{-1}$ w.m. 1, $x_1=-(v-g)\gamma_2^{1/2}1_g$ and $x_2=g\gamma_1^{1/2}1_{v-g}$;
	and $\mu_4 = 0$ w.m. 1, $x_1=\gamma_1^{1/2}1_g$ and $x_2=\gamma_2^{1/2}1_{v-g}$.
	
	Therefore, the non-zero eigenvalues of $C_Q(w_\gamma)$ are $\lambda_1 = \frac{1-g\gamma_1}{g(v-g)}$ w.m. $v-g-1$, $\lambda_2 = \frac{\gamma_1}{v-g}$ w.m. $g-1$, $\lambda_3=\frac{\gamma_1(1-g\gamma_1)}{v-g}$ w.m. 1. Thus for $p \in (-\infty,0)$, the $\Phi_p$-optimal $w_\gamma$ is obtained by minimizing the convex function
	$$f_p(\gamma_1) = (v-g-1)\left(\frac{1-g\gamma_1}{g(v-g)}\right)^p + (g-1)\left(\frac{\gamma_1}{v-g}\right)^p + \left(\frac{\gamma_1(1-g\gamma_1)}{v-g}\right)^p.$$
	Then $f_p'(\gamma_1) = 0$ if and only if
	$$-(v-g-1)(1-g\gamma_1)^{p-1} + (g-1)(g\gamma_1)^{p-1} + (1-2g\gamma_1)(g\gamma_1)^{p-1}(1-g\gamma_1)^{p-1}=0 $$
	which is equivalent to
	$$-(v-g-1)(g\gamma_1)^{1-p} + (g-1)(1-g\gamma_1)^{1-p} + 1-2g\gamma_1 = 0.$$
	Using $\gamma=g\gamma_1$, we obtain (\ref{eCwControls}).
	
	If we set $p=0$ in (\ref{eCwControls}), we obtain $\gamma=g/v$, which means that $w_\gamma$ is a uniform design. Such design is indeed $D$-optimal, because it is well known that the uniform design is $D$-optimal for any system of contrasts.
	
	The smallest non-zero eigenvalue of $C_Q(w_\gamma)$ is $\lambda_3$ 
	and hence the $\Phi_{-\infty}$-optimal design can be obtained by maximizing 
	$$f_{-\infty}(\gamma) = \frac{\gamma(1-\gamma)}{g(v-g)}$$
	which has maximum in $\gamma=\frac{1}{2}$.
	
	Note that even in the case $g=1$, where $Q$ is not rank deficient, the eigenvalues of $V(w)$ are inverses of the eigenvalues of $N_Q(w)$ and thus our results hold.
\qed


\paragraph{Proof of Theorem \ref{tMVopt}\\}
	This proof will closely follow the proof of Theorem \ref{tCwControls}.
	The covariance matrix of the least-square estimators is proportional to $V(w) = Q^TM^{-1}(w)Q$. Note that since the $MV$-optimality criterion $\Phi_{MV}$ depends only on the diagonal of the variance matrix, it is permutationally invariant.
	
	Let $w$ be a treatment proportions design and let $\tilde{P}$, $\tilde{w}$, $\gamma_1$, $\gamma_2$, $\gamma$ and $w_\gamma$ be defined as in the proof of Theorem \ref{tCwControls}. Then $Q^T\tilde{P} = BQ^T$, where $B=P_1 \otimes P_2$, which is a permutation matrix. Thus $V(\tilde{P}w) = BV(w)B^T$, $\Phi_{MV}(\tilde{P}w) = \Phi_{MV}(w)$ and $\Phi_{MV}(\tilde{w}) \leq \Phi_{MV}(w)$. Similarly, $\Phi_{MV}(w') \leq \Phi_{MV}(w)$. It follows that an optimal design exists in the class of designs $w_\gamma$. 
	
	We have $V(w_\gamma)= \gamma_1^{-1}I_g \otimes J_{v-g} + \gamma_2^{-1} J_g \otimes I_{v-g}$ and all its diagonal elements are $\gamma_1^{-1} + \gamma_2^{-1}$. Thus the optimal $\gamma_1$ may be obtained by minimizing
	$$f_{MV}(\gamma_1) = \gamma_1^{-1} + \frac{v-g}{1-g\gamma_1},$$
	which has minimum in $\gamma_1^* = \frac{\sqrt{g(v-g)}-g}{g(v-2g)}$, thus $\gamma^* = \frac{\sqrt{g(v-g)}-g}{v-2g}$.
\qed

\begin{lemma}\label{lTrigBalance}
	Let $l,m \in \mathbb{N}$, let $\xi_p$ be an exact design of size $l$ and let $\xi=\xi_p\xi_p...\xi_p$ be the exact design of size $n=lm$ formed by an $m$-fold replication of $\xi_p$. Assume that $a \in \mathbb{N}$ is not an integer multiple of $m$. Then, $\xi$ is balanced for the nuisance regressors of the form $\cos(a\phi_nt)$ and $\sin(a\phi_nt)$, $t=1,...,n$. 
\end{lemma}
\begin{proof}
	Let $u \in \{1,...,v\}$. Using the fact that $\xi(u,k+lj)=\xi_p(u,k)$ for all $k \in \{1,...,l\}$ and $j \in \{0,...,m-1\}$, we obtain
	\begin{eqnarray*}
		\sum_{t=1}^n \xi(u,t)\cos(a\phi_nt)+\mathrm{i} \sum_{t=1}^n \xi(u,t)\sin(a\phi_nt) = \sum_{t=1}^n \xi(u,t)e^{a\phi_nt\mathrm{i}} \\ = \sum_{j=0}^{m-1}\sum_{k=1}^l\xi(u,k+lj)e^{a\phi_n(k+lj)\mathrm{i}}
		=\left(\sum_{k=1}^l\xi_p(u,k)e^{a\phi_nk\mathrm{i}}\right)\left(\sum_{j=0}^{m-1}e^{(a\phi_nl \mathrm{i})j}\right).
	\end{eqnarray*} 
	Note that if $a$ is not an integer multiple of $m$ then $a\phi_nl=2\pi (a/m)$ is not an integer multiple of $2\pi$, which implies $e^{a\phi_nl\mathrm{i}} \neq 1$. In that case
	\begin{equation*}
	\sum_{j=0}^{m-1}e^{(a\phi_nl\mathrm{i})j}=\frac{1-e^{a2\pi \mathrm{i}}}{1-e^{a\phi_n l\mathrm{i}}}=0.
	\end{equation*}
\end{proof}


\paragraph{Proof of Proposition \ref{pTrig}\\}
	The proposition follows from Theorem \ref{tBalanceOpt} and Lemma \ref{lTrigBalance}.
\qed


\paragraph{Proof of Proposition \ref{pSimplex}\\}
	
	It is well known that a point $x$ is a vertex of the set $\{x|Ax=b,x\geq0\}$ if and only if the system $\{A_j| x_j > 0\}$, where $A_j$ is the $j$-th column of $A$, has full rank.
	
	The matrix $A$ consists of $v+(v-1)d+n$ rows, but they are linearly dependent.
	Let $k$ be the affine dimension of $\{h(t)\}_{t \in \T}$ and, without the loss of generality, let $\T=\{1,\ldots,n\}$. Then, the matrix $[h(2)-h(1), \ldots, h(n)-h(1)]$ has rank $k$ and thus its row space has dimension $k$. That is, without the loss of generality, we obtain that
	$h_i(t)-h_i(1) = \sum_{j=1}^k c_j^{(i)}\big(h_j(t)-h_j(1)\big)$ for some $ c_1^{(i)}, \ldots,  c_k^{(i)} \in \mathbb{R}$, for $i>k$ and $t\in \{1,\ldots,n\}$ (for $t=1$, we formally get $0=0$). Let $u \in \{1,\ldots,v\}$. Then, if (ii) is satisfied in the first $k$ coordinates of $h$, i.e., for $h_1, \ldots, h_k$, we have for all $i>k$ and $u \in \{1,...,v\}$
	$$\begin{aligned}
	w_1^{-1} \sum_t \xi(1,t) h_i(t) 
	&=  w_1^{-1} \Big(h_i(1) -  \sum_{j=1}^k c_j^{(i)}h_j(1)\Big)\sum_t \xi(1,t) + \sum_{j=1}^k c_j^{(i)}w_1^{-1} \sum_t \xi(1,t)h_j(t) \\
	&= h_i(1) -  \sum_{j=1}^k c_j^{(i)}h_j(1) + \sum_{j=1}^k c_j^{(i)}w_u^{-1} \sum_t \xi(u,t)h_j(t) \\
	&=  w_u^{-1} \sum_t \xi(u,t)\Big(h_i(1) -  \sum_{j=1}^k c_j^{(i)}h_j(1)\Big) + \sum_{j=1}^k c_j^{(i)}w_u^{-1} \sum_t \xi(u,t)h_j(t) \\
	&= w_u^{-1}\sum_t \xi(u,t) \Big[ \Big(h_i(1) -  \sum_{j=1}^k c_j^{(i)}h_j(1)\Big) +  \sum_{j=1}^k c_j^{(i)}h_j(t)\Big] \\
	&= w_u^{-1} \sum_t \xi(u,t) h_i(t),
	\end{aligned}$$
	where the second and the third equality hold because of (i).
	It follows that (ii) provides at most $k(v-1)$ additional linearly independent equalities.
	
	If $\xi$ satisfies (i), it holds that $\sum_{u,t} \xi(u,t)=1$. Thus, if $\xi$ satisfies (iii) for $t=1, \ldots, n-1$, we have $1= \sum_{t=1}^{n-1} \sum_u \xi(u,t) + \sum_u \xi(u,n) = \frac{n-1}{n} +  \sum_u \xi(u,n)$ and therefore (iii) holds also for $t=n$. That is, (iii) provides only $n-1$ additional linearly independent equalities. Hence, the rank of $A$ is at most $v + (v-1)k + n-1$ and a vertex $x$ contains at most $v + (v-1)k + n-1$ support points.
	
\qed

\bibliographystyle{plainnat}
\bibliography{nuis_res}

\end{document}